\documentclass[a4,11pt]{amsart}

\usepackage[misc]{ifsym}
\usepackage{latexsym,paralist}
\usepackage{amsmath}
\usepackage{amssymb}
\usepackage{amsthm}
\usepackage[mathcal]{eucal}
\usepackage{amsmath, amssymb,enumerate}
\usepackage{mathrsfs}
\usepackage{enumerate}
\usepackage[colorlinks=true, pdfstartview=FitV, linkcolor=blue, citecolor=blue,
 urlcolor=blue]{hyperref}
\usepackage{color}
\allowdisplaybreaks

\usepackage[pagewise]{lineno}

\makeatletter \@addtoreset{equation}{section}

\makeatother \makeatletter

\newtheorem{theorem}{Theorem}
\newtheorem{proposition}{Proposition}
\newtheorem{lemma}{Lemma}
\newtheorem{remark}{Remark}

\newtheorem{coro}{Corollary}

\newcommand{\cc}{C_{c}^{\infty}(\R^{N})}

\newcommand{\R}{\mathbb{R}}

\newcommand{\C}{\mathbb{C}}

\newcommand{\N}{\mathbb{N}}

\newcommand{\Rp}{\textrm{\emph{Re}\,}}

\allowdisplaybreaks
\begin{document}
\title[Fourth-order operators with unbounded coefficients]{Fourth-order operators with unbounded coefficients}

\email{fgregorio@unisa.it}
\email{chiara.spina@unisalento.it} 
\email{ctacelli@unisa.it}
\thanks{The authors are members of the Gruppo Nazionale per l’Analisi Matematica, la Probabilità
e le loro Applicazioni (GNAMPA) of the Istituto Nazionale di Alta Matematica (INdAM). This article is based upon work from COST Action CA18232 MAT-DYN-NET, supported by COST (European Cooperation in Science and Technology), www.cost.eu.}
\thanks{$^*$Corresponding author: Federica Gregorio}


\maketitle

\centerline{\scshape
Federica Gregorio$^{
*1}$, Chiara Spina$^{2}$, Cristian Tacelli$^{1}$}

\medskip

{\footnotesize
 \centerline{$^1$Dipartimento di Matematica, Università degli Studi di Salerno, Fisciano, Italy}
} 

\medskip

{\footnotesize
 \centerline{$^2$Dipartimento di Matematica e Fisica "Ennio De Giorgi", Università del Salento, Lecce, Italy}
}

\begin{abstract}
We prove that  operators of the form $A=-a(x)^2\Delta^{2}$, with $|D a(x)|\leq c a(x)^\frac{1}{2}$,  generate  analytic semigroups in $L^p(\R^N)$ for $1<p\leq\infty$ and in $C_b(\R^N)$. In particular, we deduce generation results for the operator $A :=- (1+|x|^2)^{\alpha} \Delta^{2}$, $0\leq\alpha\leq2$.  Moreover, we characterize the maximal domain of such operators in $L^p(\R^N)$ for $1<p<\infty$.
\end{abstract}

\section{Introduction}
There exists nowadays a huge literature about second order elliptic operators with unbounded coefficients and, in particular, polynomial growing coefficients, see for example \cite{met-oka-sob-spi,bcgt} and the references therein. One of the first studied operators in this  context has been $L=(1+|x|^\alpha)\Delta$, $\alpha>0$.  We notice that methods and results are different in the cases $\alpha\leq2$ or $\alpha>2$, cf. \cite{for-lor,met-spi2}. 

\medskip
On the other hand, recently, the interest towards elliptic operators of higher order has increased considerably: they appear, for instance, in models of elasticity \cite{mel}, free boundary problems \cite{ada} and non-linear elasticity \cite{ant}.
The semigroup generated by a class of higher order elliptic operators with  bounded measurable coefficients in $L^p(\R^N)$ has been systematically studied by Davies, cf. \cite{dav95a}. Unbounded coefficients for higher order operators have also been considered in recent times. In \cite{AGRT} the authors obtained generation results for the square of the Kolmogorov operator $L=\Delta+\frac{\nabla\mu}{\mu}\cdot\nabla$  in the weighted space $L^2(\R^N,d\mu)$,  giving also sufficient conditions on the measure $\mu$ for the characterization of both the domains of $L$  and $-L^2$. In \cite{gre-tac} polynomial growing coefficients are considered. More precisely, the operator $A=(1+|x|^\alpha)^2\Delta^2+|x|^{2\beta}$ with $\alpha>0$ and $\beta>(\alpha-2)^+$ has been considered. By using form methods, it has been proved that $(-A,D(A))$ is the   generator of an analytic $C_0$-semigroup in the space $L^2(\R^N)$ for $N\geq5$ and the domain has been characterised when $N>8$. The constraints on the dimension are due to the application of  Rellich's inequalities.
   
\medskip

In this paper we focus on elliptic and parabolic solvability in   $L^p$ spaces for every $p\in(1,\infty]$ of the operator $$A :=- (1+|x|^2)^{\alpha} \Delta^{2}$$ 
in the case $ 0\leq\alpha\leq2$. More generally, our methods enable us to consider operators of the form $A=-a(x)^2\Delta^2$ with suitable assumptions on the growth of the coefficient $a$.
It is well known that maximum principles for higher order differential equations fail, and hence the positivity preserving property of the semigroup is not assured. Moreover, the classical symmetric Markov semigroup theory is
not applicable, therefore there is no a priori reason why the generated semigroup on $L^2(\R^N)$ should be extendable
to further $L^p$ spaces. In this respect, another feature of higher order operators is the lack of dissipativity in $L^p(\R^N)$ when $p\neq2$, cf. \cite{lan-maz}. Therefore, one cannot expect to obtain the generation of a semigroup by means of dissipativity in $L^p$ for every $1<p<\infty$.
We overcome the problem  by proving that the operator $A$, endowed with the domain
$$D(A_p)=\{u\in W_{\rm loc}^{4,p}(\R^N)\cap L^p(\R^N):\ a^\frac{1}{2}D u,\ a D^2u,\ a^\frac{3}{2}D^3 u,\ a^2D^4 u\in L^p(\R^N)\}$$
is sectorial for every $1<p<\infty$.  
In particular, the paper is organised as follows. In Section 2, through a localisation argument in balls with variable radius, as in \cite{for-lor}, we prove  a priori estimates of the form
\[\|a^\frac12D u\|_p+\|a D^2u\|_p+\|a^\frac32D^3u\|_p+\|a^2D^4u\|_p\leq C\left(\|Au\|_p+\|u\|_p\right)\]
in $D(A_p)$. In view of proving sectoriality following Agmon's ideas (see \cite[Theorem 2.1]{agm62}), by suitably freezing the coefficients,  we provide a priori estimates also for the  time dependent operator $A_1=-a^2(x)\Delta^2-e^{i\theta}D_t^{(4)},\ \theta\in\left[-\frac12,\frac12\right]$.
Furthermore, making use of interior estimates, at the end of Section 2, we characterize the maximal domain by proving that $D(A_p)$ coincides with $D(A_{p, \rm max})=\{u\in W^{4,p}_{\rm loc}(\R^N)\cap L^p(\R^N):\ u,\  Au\in L^p(\R^N)\}$.  

\smallskip
In Section 3 we deduce the generation results in $L^p$, $1<p<\infty$. By arguing as in \cite[Theorem 2.1]{agm62}, thanks to the a priori estimates for $A_1$, we prove the resolvent estimate 
\[|\lambda|\| u\|_p\leq C\|\lambda u-Au\|_p\] for $\Rp\lambda$ large enough. After an approximation precedure, this estimate leads to the solvability of the elliptic problem $\lambda u-Au =f$ for $\Rp\lambda>r_p$ for some suitable $r_p>0$ and then, by classical results, to the generation of an analytic semigroup.

\smallskip
Finally, Section 4 deals with generation results in the spaces $C_b(\R^N)$ and $L^\infty(\R^N)$ via sectoriality proved by means of Stewart-Masuda localization technique.

\medskip

 \textbf{Notations.} We use standard notations for function spaces. We denote by $L^p(\R^N)$ and $W^{k,p}(\R^N)$ the standard Lebesgue and Sobolev spaces, respectively. $C_c^\infty(\R^N)$ is the space of test functions and $C_b(\R^N)$ is the space of bounded continuous functions. We denote by $W_{\rm comp}^{4,p}(\R^{N})$   the space of functions in $W^{4,p}(\R^N)$ with compact support.
$B_R$ denotes the open ball of radius $R$ centred at 0, whereas $B_R(x_0)$ is the ball centred at $x_0$ for some $x_0\in\R^N$.  We denote by $B_R^c$ the complementary set of the open ball, $B_R^c=\R^N\setminus B_R$. Moreover, we use the symbol $\|\cdot\|_{p,R}$ to shorten the notation $\|\cdot\|_{L^p(B_R(x_0))}$. In the proofs, $C$ is a  positive constant that may vary from line to line. Finally, for every function $u$ smooth enough we set
\begin{align*}|D^4u|=\left(\sum_{i,j,k,l=1}^N|D_{ijkl}u|^2\right)^\frac12&, \ |D^3u|=\left(\sum_{i,j,k=1}^N|D_{ijk}u|^2\right)^\frac12\\|D^2u|=\left(\sum_{i,j=1}^N|D_{ij}u|^2\right)^\frac12&, \ |Du|=\left(\sum_{i=1}^N|D_{i}u|^2\right)^\frac12.\end{align*}

\section{A priori estimates and domain characterization}

Let us consider operators of the form  $A=-a(x)^2\Delta^2$. Our standing assumptions on the coefficient of $A$ are
\begin{equation} \label{gradient}\tag{H} a \in C^1(\R^N),\quad
 a(x)>0\ {\rm for\ all}\ x\in\R^N,\quad |D a(x)|\leq c a(x)^\frac{1}{2},
\end{equation}
for some positive constant $c$.
The  last inequality   will be crucial to get the main results of the paper and it is equivalent to $\|D a^\frac{1}{2}\|_\infty\leq \frac{c}{2}$.

\begin{remark}
Observe that, if $0\leq\alpha\leq2$, $a(x)=(1+|x|^2)^{\frac{\alpha}{2}}$ satisfies (\ref{gradient}), indeed
\begin{equation*}  
|D a(x)|=\alpha (1+|x|^2)^\frac{\alpha}{4}(1+|x|^2)^{\frac{\alpha}{4}-1}|x|\leq \alpha a(x)^\frac{1}{2}.
\end{equation*}
\end{remark}
We start by proving some a priori estimates in 
$$D(A_p)=\{u\in W_{\rm loc}^{4,p}(\R^N)\cap L^p(\R^N):\ a^\frac{1}{2}D u,\ a D^2u,\ a^\frac{3}{2}D^3 u,\ a^2D^4 u\in L^p(\R^N)\}$$ for the operator $A=- a(x)^2 \Delta^{2}$ and for an auxiliary operator which will be later introduced. 
We need  a density result, useful to work with smooth functions, and a covering result to work with localized a priori estimates.
\begin{lemma}
Let  $1<p<\infty$ and assume that $a$ satisfies (\ref{gradient}). Then $C_c^\infty(\R^N)$ is dense in $D(A_p)$ with respect to the norm 
$$\|u\|_{D(A_p)}=\|u\|_p+\|a^\frac{1}{2}D u\|_p+\|a D^2u\|_p+\|a^\frac{3}{2}D^3 u\|_p+\|a^2D^4 u\|_p.$$
\end{lemma}
\begin{proof} 
We first observe that if we take $u\in W^{4,p}(\R^N)$ with compact support, then we can approximate it in the norm of $D(A_p)$ with smooth functions by convolution with standard mollifiers. Therefore it remains to prove that a function in $D(A_p)$ can be approximated by a sequence of functions with compact support.
Let $u\in D(A_p)$ and set $u_n=u\varphi_n$ with $\varphi_n\in C^\infty_c(\R^N)$ such that 
\[\begin{cases} \varphi_n=0 \quad \textrm{in}\ B_{2n}^c,\\ \varphi_n=1\quad \textrm{in}\ B_n,\\0\leq\varphi_n\leq1,\\|D^h\varphi_n|\leq {C}{(1+|x|)^{-h}}\quad \textrm{for}\ h=1,2,3,4.
\end{cases}\]
Observe that such functions exist. Indeed, one can consider a function $\varphi\in C^\infty([0,\infty))$ such that $0\leq\varphi\leq1, \varphi(t)=1$ for $t\in[0,1]$ and $\varphi(t)=0$ for $t\geq2$ and then set 
\[\varphi_n(x)=\begin{cases} 1 & \textrm{in}\ B_n,\\\varphi\left(\frac{|x|}{n}\right) & \textrm{in}\ B_n^c.\end{cases}\]
Moreover, by assumption (\ref{gradient}), $a(x)\leq C(1+|x|^2)$ for some positive constant $C$ and for every $x\in\R^N$ and then
\begin{equation}\label{derphi}
|D^h\varphi_n|\leq Ca^{-\frac{h}{2}}.
\end{equation}

We  show   that $u_n$ converges to $u$ with respect to the norm $||\cdot||_{D(A_p)}$. One has $u_n\to u\in L^p(\R^N)$ by dominated convergence. For the first order term we observe that
\begin{align*}
|a^\frac12D u_n-a^\frac12D u|&= a^\frac12|D u\, \varphi_n+uD\varphi_n-D u|
\\&\leq  (1-\varphi_n)a^\frac12|D u|+  a^\frac12 |u||D\varphi_n| 
\end{align*}
and the right-hand side of the inequality converges to 0 a.e. Moreover,  in view of \eqref{derphi},
\[(1-\varphi_n)a^\frac12|D u|+ a^\frac12 |u||D\varphi_n|\leq  (1-\varphi_n)a^\frac12|D u|+C|u|\chi_{B_{2n}\setminus B_n}\in L^p(\R^N).\]
Therefore $a^\frac12D u_n$ converges to $a^\frac12D u$  in $L^p(\R^N)$.
As regards the other terms, for $h=2,3,4$, one has 
\begin{align*}
|a^{\frac{h}{2}} D^h u_n-a^{\frac{h}{2}} D^h u|&\leq  (1-\varphi_n)a^{\frac{h}{2}} |D^h u|+C \sum_{i=0}^{h-1}\sum_{\substack{|l|=i\\|m|=h-i}}  a^{\frac{h}{2}}|D^lu||D^m\varphi_n|\to0 \ \textrm{a.e.}
\end{align*}
and by \eqref{derphi}
\begin{align*}(1-\varphi_n)a^{\frac{h}{2}} |D^h u|+C& \sum_{i=0}^{h-1}\sum_{|l|=i,|m|=h-i}  a^{\frac{h}{2}}|D^lu||D^m\varphi_n|\leq\\ &\qquad(1-\varphi_n)a^{\frac{h}{2}} |D^h u|+C\sum_{i=0}^{h-1}\sum_{\substack{|l|=i\\|m|=h-i}}  a^{\frac{h-m}{2}}|D^lu|\chi_{B_{2n}\setminus B_n}\in L^p(\R^N).\end{align*}
\end{proof}

The  covering result with variable radius balls as formulated here  is proved  in \cite{cup-for}.
\begin{proposition}\cite[Proposition 6.1]{cup-for} \label{pr:covering}
Let ${\mathcal F} =\{  B_{\rho (x)}(x)\}_{x\in \R^{N}}$, 
where $\rho : \R^{N} \to \R_{+}$ is a Lipschitz continuous function with Lipschitz
constant $\kappa$ strictly less than $\frac 12$.
Then, there exist a countable subcovering $\{ B_{\rho(x_{n})}(x_{n})\}_{n\in \N}$
 and a natural number $\xi=\xi (N,\kappa)$ such that at most $\xi$ 
among the doubled balls $\{ B_{2\rho(x_{n})}(x_{n} )\}_{n\in \N}$ overlap.
\end{proposition}

Observe that by the inequality (\ref{gradient}) satisfied by $a(x)$, the function $\rho(x)=\frac{\eta}{2c}a(x)^\frac{1}{2}$, $0<\eta\leq 1$, satisfies the assumptions of the previous covering Proposition with Lipschitz constant smaller than $\frac{\eta}{4}$. 

The next lemma constitutes a first step toward the a priori estimates.
\begin{lemma}  \label{interpolation}
Let  $1<p<\infty$ and assume that $a$ satisfies (\ref{gradient}). Then there exists a positive constant  $C=C(N,p,c)$ such that for every $\varepsilon>0$ the following inequalities hold
\begin{align}
&\|a^{\frac{1}{2}}Du\|_{p}\leq \varepsilon ^{3}\|a^{2}D^{4}u\|_{p}+\frac{C}{\varepsilon}\|u\|_{p} \label{eq:interp00-01}\\
&\|aD^{2}u\|_{p}\leq \varepsilon^2 \|a^{2}D^{4}u\|_{p}+\frac{C}{\varepsilon^2}\|u\|_{p}\label{eq:interp00-02} \\
&\|a^{\frac{3}{2}}D^{3}u\|_{p}\leq \varepsilon \|a^{2}D^{4}u\|_{p}+\frac{C}{\varepsilon^{3}}\|u\|_{p} \label{eq:interp00-03}\\
& \|a^2D^4u\|_p\leq C\left( \|Au\|_p+\|u\|_p\right) \label{eq:calderon-zigmund}
\end{align}
for every $u\in C_{c}^{\infty}(\R^{N})$. 
\end{lemma}
\begin{proof}  First  we prove for $h\in \N$ 
\begin{equation}\label{eq:interp1}
\|a^{\frac{h}{2}}D^{h}u \|_{p}\leq C \|a^{\frac{h+1}{2}}D^{h+1} u\|^{\frac{1}{2}}_{p}\|a^{\frac{h-1}{2}}D^{h-1}u\|^{\frac{1}{2}}_{p}.
\end{equation}
To this purpose we show that for $v\in C_{c}^{\infty}(\R^{N})$
\begin{equation}\label{eq:interp0}
\|a^{\frac{h}{2}}D v\|_{p}\leq C \|a^{\frac{h+1}{2}}D^2 v\|^{\frac 12}_{p}\|a^{\frac{h-1}{2}}v\|^{\frac 12}_{p},
\end{equation}
from which \eqref{eq:interp1} easily follows replacing $v$ by     $D^{h-1}u$.

Let us set $\rho(x)=\frac{1}{2c}a(x)^\frac{1}{2}$ and note that, as stated before,  by \eqref{gradient} it follows that $\rho(x)$ satisfies the assumptions of Proposition \ref{pr:covering} with Lipschitz constant smaller than $\frac{1}{4}$. We fix  $x_{0}\in \R^{N}$ and set $R=\rho (x_{0})=\frac{1}{2 c}a^{\frac{1}{2}}(x_{0})$. 
 Then, if $x\in B_{2R}(x_0)$ the following inequalities hold
\[
\frac{1}{2}a(x_{0})^{\frac{1}{2}}\leq a(x)^{\frac{1}{2}}\leq \frac{3}{2}a(x_{0})^{\frac{1}{2}}
\]
and consequently 
\[
\left (\frac{1}{2}\right)^{ k}a(x_{0})^{\frac{k}{2}}\leq a(x)^{\frac{k}{2}}\leq \left (\frac{3}{2}\right)^{k}a(x_{0})^{\frac{k}{2}}
\]
for $k=h-1,h,h+1$.

Let now  $\vartheta\in C_c^{\infty }(\R^N)$ be such that $0\leq \vartheta \leq 1$,
$\vartheta(x)=1$ for $x\in B_{R}(x_{0}),\vartheta(x)=0$ for $x\in B^{c}_{2R}(x_{0})$ and $|D^h\vartheta|\leq {C}{R^{-h}}$ for $h=1,2,3,4$.

By the classical interpolation inequality (see \cite[Proposition 1.3.8]{lor-rha-book}), we have
\begin{align*}
\|a^{\frac{h}{2}}D v\|_{p, R}
&\leq C\|a(x_{0})^{\frac{h}{2}}D v\|_{p, R}\leq  C\|a(x_{0})^{\frac{h}{2}}D (\vartheta v)\|_{p}\\
& \leq C  \| a(x_{0})^{\frac{h+1}{2}} D^{2} (\vartheta v)\|_{p}^{\frac{1}{2}}  \| a(x_{0})^{\frac{h-1}{2}} \vartheta v\|^{\frac{1}{2}}_{p}\\
& \leq \varepsilon  \|a(x_{0})^{\frac{h+1}{2}} D^{2} (\vartheta v)\|_{p}
		+\frac{C}{\varepsilon } \| a(x_{0})^{\frac{h-1}{2}} \vartheta v\|_{p}\\
&  \leq \varepsilon  \|a(x_{0})^{\frac{h+1}{2}} D^{2} v\|_{p,2R}+\frac{C}{R}\varepsilon \|a(x_{0})^{\frac{h+1}{2}} D v\|_{p,2R}
		+\frac{C}{R^{2}}\varepsilon \|a(x_{0})^{\frac{h+1}{2}} v\|_{p,2R}\\
&\quad		+\frac{C}{\varepsilon}\| a(x_{0})^{\frac{h-1}{2}} v\|_{p,2R}.
\end{align*}

Then
taking into account that 
$\frac{1}{R}a(x_{0})^{\frac{h+1}{2}}=2c a(x_{0})^{\frac{h}{2}}\leq Ca(x)^{\frac{h}{2}}$
and $\frac{1}{R^{2}}a(x_{0})^{\frac{h+1}{2}}=4c^2 a(x_{0})^{\frac{h-1}{2}}\leq Ca(x)^{\frac{h-1}{2}}$ in $B_{2R}(x_0)$ one has
\begin{align}\label{eq:cover-x00}
 \|a^{\frac{h}{2}}D v\|_{p,R} \leq C \left( \varepsilon \|a^{\frac{h+1}{2}}D^2 v\|_{p,2R}+\varepsilon \|a^{\frac{h}{2}}D v\|_{p,2R}
		+\left(\varepsilon+\frac{1}{\varepsilon}\right)\|a^{\frac{h-1}{2}}v\|_{p,2R}
		\right).
\end{align}

Let $\{B_{\rho(x_n )}(x_n)\}$ be a countable covering of $\R^N$ as in Proposition \ref{pr:covering}
such that at most $\zeta $ among the double balls $\{B_{2\rho(x_n )}(x_n)\}$ overlap.

We write \eqref{eq:cover-x00} with $x_0$ replaced by $x_n$ and sum over $n$.
Taking into account the above covering result, we get
$$
\|a^{\frac{h}{2}}D v\|_{p}\leq 
C \left( \varepsilon \|a^{\frac{h+1}{2}}D^2 v\|_{p}+\varepsilon \|a^{\frac{h}{2}}D v\|_{p}
		+\left(1+\frac{1}{\varepsilon}\right)\|a^{\frac{h-1}{2}}v\|_{p}\nonumber
		\right).
$$
Choosing $\varepsilon'=\dfrac{C\varepsilon}{1-C\varepsilon}$   one can find a suitable positive $C$ such that  
$$
\|a^{\frac{h}{2}}D v\|_{p}\leq 
\varepsilon' \|a^{\frac{h+1}{2}}D^2 v\|_{p}
+\dfrac{C}{\varepsilon'}\|a^{\frac{h-1}{2}}v\|_{p}\nonumber
$$
from which 
\eqref{eq:interp0} follows.

In particular, we have proved the following inequalities
\begin{align}
\|a^{\frac{1}{2}}Du\|_p
	&\leq C\|aD^2u\|^\frac{1}{2}\|u\|^\frac{1}{2}_p \label{eq:inter1_01}\\
\|aD^2u\|_p
	&\leq C\|a^{\frac 32}D^3u\|^\frac{1}{2}_p\|a^{\frac 12}Du\|^\frac{1}{2}_p\label{eq:inter1_02}\\
\|a^{\frac{3}{2}}D^3u\|_p
	&\leq C\|a^2D^4u\|^\frac{1}{2}_p\|aD^2u\|^\frac{1}{2}_p. \label{eq:inter1_03}
\end{align}
Combining \eqref{eq:inter1_01} and \eqref{eq:inter1_03} with \eqref{eq:inter1_02}
we obtain
\[
\|aD^2u\|_p\leq C
\|aD^2u\|^\frac{1}{4}_p\|u\|^\frac{1}{4}_p
\|a^2D^4u\|^\frac{1}{4}_p\|aD^2u\|_p^\frac{1}{4},
\]
from which we have
\begin{equation}\label{eq:inter1_04}
\|aD^2u\|_p\leq C\|a^2D^4u\|_p^\frac{1}{2} \|u\|_p^\frac{1}{2}
\end{equation} 
that leads to \eqref{eq:interp00-02}.
Combining now \eqref{eq:inter1_04} with \eqref{eq:inter1_01} 
and  applying Young's inequality we  have
\[
\|a^{\frac{1}{2}}Du\|_p\leq C
\|a^2D^4u\|_p^\frac{1}{4} \|u\|_p^\frac{3}{4}
\leq  \varepsilon^4\|a^2D^4u\|_p+\frac{C}{\varepsilon^\frac{4}{3}}\|u\|_p
\]
from which \eqref{eq:interp00-01} follows with $\varepsilon'=\varepsilon^{\frac43}$.

Combining \eqref{eq:inter1_04} with \eqref{eq:inter1_03}  and 
applying Young's inequality we have
\[
\|a^{\frac{3}{2}}D^3u\|_p\leq C
\|a^2D^4u\|_p^\frac{3}{4} \|u\|_p^\frac{1}{4}
\leq  \varepsilon^{\frac{4}{3}}\|a^2D^4u\|_p+\frac{C}{\varepsilon^4}\|u\|_p
\]
from which \eqref{eq:interp00-03} follows with $\varepsilon'=\varepsilon^{\frac43}$.

Finally, in order to prove \eqref{eq:calderon-zigmund}, we  fix $x_0\in \R^N$, $R=\rho(x_0)$ and $\vartheta\in C_c^\infty(\R^N)$ as before.  We have
\begin{align*}
\|a^2D^4u\|_{p,R}&\leq Ca(x_0)^2\|D^4(u\vartheta)\|_{p,2R}
	\\
& \leq 
	Ca(x_0)^2\|\Delta^2(u\vartheta)\|_{p,2R}\\
& \leq 
	Ca(x_0)^2\|\Delta^2 u \|_{p,2R}
	+C\Bigg(\frac{a(x_0)^2}{R}\|D^3u\|_{p,2R}
	\\
&\quad+ \frac{a(x_0)^2}{R^2}\|D^2u\|_{p,2R}	
	+  \frac{a (x_0)^2}{R^3}\|Du\|_{p,2R}
	+  \frac{a (x_0)^2}{R^4}\|u\|_{p,2R}
	\Bigg)\\
&\leq C\left (
	\|a^2\Delta ^2u \|_{p,2R}
	+\|a^{\frac{3}{2}}D ^3u\|_{p,2R}+\|aD ^2u\|_{p,2R}
	+\|a^{\frac{1}{2}}Du\|_{p,2R}
	+\|u\|_{p,2R}
	\right ).
 \end{align*}
By the above covering argument one has
\begin{align*}
\|a^2D^4u\|_p &\leq C\left (
	\|a^2\Delta ^2u \|_{p}
	+\|a^{\frac{3}{2}}D ^3u\|_{p}+\|aD ^2u\|_{p}
	+\|a^{\frac{1}{2}}Du\|_{p}
	+\|u\|_{p}
	\right )\\
&\leq 	C\left (
	\|a^2\Delta ^2u \|_{p} +\varepsilon \|a^2D^4u\|_p
	+\frac{C}{\varepsilon}\|u\|_p\right).
\end{align*}
Therefore, \eqref{eq:calderon-zigmund} holds.
\end{proof}

\begin{remark} \label{supp-compatto}
The inequalities proved in Lemma \ref{interpolation} remain true for functions in $W^{4,p}_{\rm comp}(\R^N)$.
\end{remark}

The following proposition is an immediate  consequence of the previous lemma and gives the announced a priori estimates.

\begin{proposition} \label{domain}
Let  $1<p<\infty$ and assume that $a$ satisfies (\ref{gradient}). Then, for $u\in D(A_p)$ the following inequality holds
$$\|a^\frac{1}{2}D u\|_p+\|a D^2u\|_p+\|a^\frac{3}{2}D^3 u\|_p+\|a^2D^4 u\|_p\leq C(\|Au\|_p+\|u\|_p)$$
where the  constant $C$ depends on $N,p$ and $c$. 
\end{proposition}

To prove the sectoriality estimate for the resolvent as in \cite[Theorem 2.1]{agm62}, we need a priori estimates also for the elliptic operator $A_1=-a^2(x)\Delta^2-e^{i\theta}D_t^{(4)}$, $\theta\in \left[-\frac{\pi}{2},\frac{\pi}{2}\right]$, in $L^p(\R^{N+1})$. 
\begin{proposition} \label{a-prioriAux}
Let  $1<p<\infty$ and assume that $a$ satisfies (\ref{gradient}). Then there exists a positive constant $C=C(N,p,c)$ such that
\begin{align*}
\sum_{k=1}^4 \|D^{(k)}_tu\|_{L^p(\R^{N+1})}&+\|a^\frac{1}{2}D u\|_{L^p(\R^{N+1})}+\|a D^2u\|_{L^p(\R^{N+1})}+\|a^\frac{3}{2}D^3 u\|_{L^p(\R^{N+1})}\\&+\|a^2D^4 u\|_{L^p(\R^{N+1})}\leq C\left(\|A_1u\|_{L^p(\R^{N+1})}+\|u\|_{L^p(\R^{N+1})}\right)
\end{align*}
for every $u\in C_c^\infty(\R^{N+1})$.
\end{proposition}
\begin{proof} Let $u\in C_c^\infty(\R^{N+1})$, $\theta\in \left[-\frac{\pi}{2},\frac{\pi}{2}\right]$.
By the  classical homogeneous a priori estimate for the operator $-\Delta^2-e^{i\theta}D^{(4)}_t$, cf. \cite[Theorem 3.2.1]{lun},
\begin{align*}
\sum_{k=1}^4 \|D^{(k)}_tu\|_{L^p(\R^{N+1})}&+\|D u\|_{L^p(\R^{N+1})}+\|D^2u\|_{L^p(\R^{N+1})}+\|D^3 u\|_{L^p(\R^{N+1})}\\&+\|D^4 u\|_{L^p(\R^{N+1})}\leq C\left(\|(\Delta^2+e^{i\theta}D^{(4)}_t)u\|_{L^p(\R^{N+1})}+\|u\|_{L^p(\R^{N+1})}\right).
\end{align*}
Fix $x_0\in \R^N$. The last inequality applied to the function $v(x,t)=u(a(x_0)^\frac{1}{2}x,t)$ gives
\begin{align} \label{a-priori1}
\sum_{k=1}^4 &\|D^{(k)}_tu\|_{L^p(\R^{N+1})}+\|a(x_0)^\frac{1}{2}D u\|_{L^p(\R^{N+1})}+\|a(x_0)D^2u\|_{L^p(\R^{N+1})}+\|a(x_0)^\frac{3}{2}D^3 u\|_{L^p(\R^{N+1})}\\&+\|a(x_0)^2D^4 u\|_{L^p(\R^{N+1})}\leq C\left(\|(a(x_0)^2\Delta^2+e^{i\theta}D^{(4)}_t)u\|_{L^p(\R^{N+1})}+\|u\|_{L^p(\R^{N+1})}\right)\nonumber
\end{align}
with a constant $C$ independent of $x_0$.
Let now $\rho(x)=\frac{\eta}{2c}a(x)^\frac{1}{2}$, with  $0<\eta<1$ to be choosen later and set $R=\rho(x_0)$. One can choose as before $\vartheta\in C_c^{\infty }(\R^N)$ such that $0\leq \vartheta \leq 1$, $\vartheta(x)=1$ for $x\in B_R(x_0)$, 
$\vartheta(x)=0$ for $x\in  B^c_{2R}(x_0)$ and $|D^h\vartheta|\leq {C}{R^{-h}}$ for $h=1,2,3,4$.  By applying (\ref{a-priori1}) to $\vartheta u$  
\begin{align*} 
\sum_{k=1}^4& \|D^{(k)}_tu\|_{L^p(B_R(x_0)\times\R)}+\|a(x_0)^\frac{1}{2}D u\|_{L^p(B_R(x_0)\times\R)}+\|a(x_0)D^2u\|_{L^p(B_R(x_0)\times\R)}\\&\quad+\|a(x_0)^\frac{3}{2}D^3 u\|_{L^p(B_R(x_0)\times\R)}+\|a(x_0)^2D^4 u\|_{L^p(B_R(x_0)\times\R)}\\&\leq C\left(\|(a(x_0)^2\Delta^2+e^{i\theta}D^{(4)}_t)(\vartheta u)\|_{L^p(\R^{N+1})}+\|\vartheta u\|_{L^p(\R^{N+1})}\right)\\&\leq 
C\left(\|(a(x_0)^2\Delta^2+e^{i\theta}D^{(4)}_t)(\vartheta u)\|_{L^p(B_{2R}(x_0)\times\R)}+\|\vartheta u\|_{L^p(B_{2R}(x_0)\times\R)}\right)
\\&\leq 
C\Big(\|(a(x)^2\Delta^2+e^{i\theta}D^{(4)}_t)(u)\|_{L^p(B_{2R}(x_0)\times\R)}+\|(a(x)^2-a(x_0)^2)\Delta^2 u\|_{L^p(B_{2R}(x_0)\times\R)}\\&\quad+\|u\|_{L^p(B_{2R}(x_0)\times\R)}\Big)+\sum_{k=1}^4\|a(x_0)^2D^{k}\vartheta D^{4-k}u\|_{L^p(B_{2R}(x_0)\times\R)}
\\&\leq 
C\Big(\|(a(x)^2\Delta^2+e^{i\theta}D^{(4)}_t)(u)\|_{L^p(B_{2R}(x_0)\times\R)}+\|(a(x)^2-a(x_0)^2)\Delta^2 u\|_{L^p(B_{2R}(x_0)\times\R)}\\&\quad+\|u\|_{L^p(B_{2R}(x_0)\times\R)}+\sum_{k=1}^4\|a(x_0)^{2-\frac{k}{2}}D^{4-k}u\|_{L^p(B_{2R}(x_0)\times\R)}\Big).
\end{align*}
By the assumption \eqref{gradient} and recalling that $R=\frac{\eta}{2c}a(x_0)^\frac{1}{2}$, we have that, for $x\in B_{2R}(x_0)$,
$$\left(1-\frac{\eta}{2}\right)a(x_0)^\frac{1}{2}\leq a(x)^\frac{1}{2}\leq \left(1+\frac{\eta}{2}\right)a(x_0)^\frac{1}{2}$$ and
\begin{align*}
|a(x)^2-a(x_0)^2|&=(a(x)+a(x_0))|a(x)-a(x_0)|\leq \eta\left(1+\frac{\eta}{2}\right)(a(x)+a(x_0))a(x_0)\\&\leq C\eta a(x_0)^2 
\end{align*}
for some positive constant $C$. Therefore $$\|(a(x)^2-a(x_0)^2)\Delta^2 u\|_{L^p(B_{2R}(x_0)\times\R)}\leq C\eta \|a(x_0)^2D^4 u\|_{L^p(B_{2R}(x_0)\times\R)}.$$
We get
\begin{align} \label{eq:cover-x0}
\sum_{k=1}^4& \|D^{(k)}_tu\|_{L^p(B_R(x_0)\times\R)}+\|a^\frac{1}{2}D u\|_{L^p(B_R(x_0)\times\R)}+\|aD^2u\|_{L^p(B_R(x_0)\times\R)}\\&\quad+\|a^\frac{3}{2}D^3 u\|_{L^p(B_R(x_0)\times\R)}+\|a^2D^4 u\|_{L^p(B_R(x_0)\times\R)}\nonumber\\&\leq 
C\Big(\|(a^2\Delta^2+e^{i\theta}D^{(4)}_t)(u)\|_{L^p(B_{2R}(x_0)\times\R)}+\|u\|_{L^p(B_{2R}(x_0)\times\R)}\nonumber\\&\quad+\sum_{k=1}^4\|a^{2-\frac{k}{2}}D^{4-k}u\|_{L^p(B_{2R}(x_0)\times\R)}+\eta \|a^2D^4 u\|_{L^p(B_{2R}(x_0)\times\R)}\Big).\nonumber 
\end{align}
Let $\{B_{\rho(x_n )}(x_n)\}$ be a countable covering of $\R^N$ as in Proposition \ref{pr:covering} 
such that at most $\zeta $ among the double balls $\{ B_{2\rho(x_n )}(x_n)\}$ overlap. 
Writing \eqref{eq:cover-x0} with $x_n$ instead of $x_0$ and summing over $n$, it follows that
\begin{align*} 
\sum_{k=1}^4& \|D^{(k)}_tu\|_{L^p(\R^{N+1})}+\|a^\frac{1}{2}D u\|_{L^p(\R^{N+1})}+\|aD^2u\|_{L^p(\R^{N+1})}\\&\quad+\|a^\frac{3}{2}D^3 u\|_{L^p(\R^{N+1})}+\|a^2D^4 u\|_{L^p(\R^{N+1})}\\&\leq 
C\left(\|(a^2\Delta^2+e^{i\theta}D^{(4)}_t)(u)\|_{L^p(\R^{N+1})}+\|u\|_{L^p(\R^{N+1})}\right)\\&\quad+\sum_{k=1}^4\|a^{2-\frac{k}{2}}D^{4-k}u\|_{L^p(\R^{N+1})}+\eta \|a^2D^4 u\|_{L^p(\R^{N+1})}.
\end{align*}
By choosing $\eta$ small enough in the choice of $\rho$ and $\varepsilon$ small enough in the interpolative inequalities in  Lemma \ref{interpolation}, we get the claimed result.
\end{proof}

We are now able to characterize the domain of $A$ in $L^p$.

\begin{proposition} \label{domain2}
Let  $1<p<\infty$ and assume that $a$ satisfies (\ref{gradient}). Then 
$$D(A_{p})=D(A_{p,\rm max})=\{u\in W^{4,p}_{\rm loc}(\R^N)\cap L^p(\R^N):\   Au\in L^p(\R^N)\}.$$
\end{proposition}
\begin{proof}
We have only to show the inclusion $D(A_{p,\rm max})\subset D(A_{p})$. 
As before, we fix $x_{0}\in \R^{N}$ and set $R=\rho (x_{0})=\frac{1}{2c}a^{\frac{1}{2}}(x_{0})$. 
Take now $\sigma\in (0,1)$ and set $\sigma':=\frac{\sigma+1}{2}$.
Consider a cutoff function $\vartheta\in \cc$ such that $0\leq\vartheta\leq 1$,
$\vartheta=1$ for $x\in B_{\sigma R}(x_{0})$, $\vartheta(x)=0$ for 
$x\in B^{c}_{\sigma'R}(x_{0})$ and $\|D^{h}\vartheta\|_{\infty}\leq {C}{R^{-h}(1-\sigma')^{-h}}$ for $h=0,1,2,3,4$.
Let $u\in W_{\rm loc}^{4,p}(\R^{N})$. Recall that by Remark \ref{supp-compatto}, the inequalites proved in Lemma \ref{interpolation} hold  for $\vartheta u\in W_{\rm comp}^{4,p}(\R^{N})$, therefore there exists a positive constant $C$ only depending on $N,p$ and $c$ such that for $h=1,2,3$ the following inequalities hold
\[
\|a^{\frac{h}{2}}D^h(\vartheta u)\|_{p}\leq \varepsilon ^{4-h}\|a^{2}\Delta^{2}(\vartheta u)\|_{p}+\frac{C}{\varepsilon^h}\|(\vartheta u)\|_{p},
\]
and for $h=4$
\[
\|a^{2}D^4(\vartheta u)\|_{p}\leq C\left( \|a^{2}\Delta^{2}(\vartheta u)\|_{p}+\|(\vartheta u)\|_{p}\right).
\]
It follows that for $h=1,2,3$ we have
\begin{align*}
 \|a^{\frac{h}{2}}D^hu\|_{p,\sigma R} &\leq\varepsilon^{4-h} \|a^{2}\Delta^{2}(u\vartheta)\|_{p,\sigma'R}+\frac{C}{\varepsilon^h}\|u\|_{p,R}\\
& \leq \varepsilon^{4-h} \|Au\|_{p,R}+\varepsilon^{4-h} Ca(x_{0})^{2}
\big( \|D^{3}u D\vartheta\|_{p,\sigma'R}+\|D^{2}uD^{2}\vartheta\|_{p,\sigma'R}\\&\quad+\|DuD^{3}\vartheta\|_{p,\sigma'R}
	+\|uD^{4}\vartheta\|_{p,\sigma'R}  \big)+\frac{C}{\varepsilon^h}\|u\|_{p,R}\\
& \leq \varepsilon^{4-h} \|Au\|_{p,R}+\varepsilon^{4-h} Ca(x_{0})^{2}
\left( \frac{1}{R(1-\sigma')}\|D^{3}u \|_{p,\sigma'R}+\frac{1}{R^{2}(1-\sigma')^{2}}\|D^{2}u\|_{p,\sigma'R}\right.\\
&\quad \left.
	+\frac{1}{R^{3}(1-\sigma')^{3}}\|Du\|_{p,\sigma'R}
	+\frac{1}{R^{4}(1-\sigma')^{4}}\|u\|_{p,\sigma'R}  \right) +\frac{C}{\varepsilon^h}\|u\|_{p,R}\\
& \leq \varepsilon^{4-h} \|Au\|_{p,R}+\varepsilon^{4-h} C
\left( \frac{a(x_{0})^{\frac{3}{2}} }{1-\sigma'}\|D^{3}u \|_{p,\sigma'R}+\frac{a(x_{0}) }{(1-\sigma')^{2}}\|D^{2}u\|_{p,\sigma'R}\right.\\
&\quad \left.
	+\frac{a(x_{0})^{\frac{1}{2}} }{(1-\sigma')^{3}}\|Du\|_{p,\sigma'R}
	+\frac{1}{(1-\sigma')^{4}}\|u\|_{p,\sigma'R}  \right)+\frac{C}{\varepsilon^h}\|u\|_{p,R}\\	
& \leq \varepsilon^{4-h} \|Au\|_{p,R}+\varepsilon^{4-h} C
\left( \frac{1}{1-\sigma'}\|a^{\frac{3}{2}}D^{3}u \|_{p,\sigma'R}+\frac{1 }{(1-\sigma')^{2}}\|aD^{2}u\|_{p,\sigma'R}\right.\\
&\quad \left.
	+\frac{1 }{(1-\sigma')^{3}}\|a^{\frac{1}{2}}Du\|_{p,\sigma'R}
	+\frac{1}{(1-\sigma')^{4}}\|u\|_{p,\sigma'R}  \right)+\frac{C}{\varepsilon^h}\|u\|_{p,R}.			
\end{align*}

For $h=4$ we obtain
\begin{align}
 \|a^{2}D^{4}u\|_{p,\sigma R }&\leq C\left(  \|Au\|_{p,R}+
 \frac{1}{1-\sigma'}\|a^{\frac{3}{2}}D^{3}u \|_{p,\sigma'R}+\frac{1 }{(1-\sigma')^{2}}\|aD^{2}u\|_{p,\sigma'R}\right.\nonumber \\
&\quad \left.
	+\frac{1 }{(1-\sigma')^{3}}\|a^{\frac{1}{2}}Du\|_{p,\sigma'R}
	+\frac{1}{(1-\sigma')^{4}}\|u\|_{p,\sigma'R} +\|u\|_{p,R} \right).\label{eq:quartordine}
\end{align}

Set $\varepsilon=\frac{1}{(2^h6C)^{\frac{1}{4-h}}}(1-\sigma')$, then one can find a suitable positive $C_1$ such that 
for $h=1,2,3$
\begin{align*}
\|a^{\frac{h}{2}}D^hu\|_{p,\sigma R}&\leq C_1(1-\sigma')^{4-h} \|Au\|_{p,R}+\frac{1}{2^h6}
\left( (1-\sigma')^{3-h}\|a^{\frac{3}{2}}D^{3}u \|_{p,\sigma'R}+(1-\sigma')^{2-h}\|aD^{2}u\|_{p,\sigma'R}\right.\\
&\qquad \left.
	+(1-\sigma')^{1-h}\|a^{\frac{1}{2}}Du\|_{p,\sigma'R}
	+(1-\sigma')^{-h}\|u\|_{p,\sigma'R}  \right)+\frac{C_1}{(1-\sigma')^h}\|u\|_{p}.\end{align*}
Multiply by $(1-\sigma')^h$, taking into account that $1-\sigma'=\frac{1-\sigma}{2}$, it follows that
\begin{align*}\label{eq:interp.01}
(1-\sigma)^h\|a^{\frac{h}{2}}D^hu\|_{p,\sigma R}&\leq C \|Au\|_{p,R}+\frac{1}{6}
\left( (1-\sigma')^{3}\|a^{\frac{3}{2}}D^{3}u \|_{p,\sigma'R}+(1-\sigma')^{2}\|aD^{2}u\|_{p,\sigma'R}\right.\nonumber \\
&\qquad \left.
	+(1-\sigma')\|a^{\frac{1}{2}}Du\|_{p,\sigma'R}
	+\|u\|_{p,\sigma'R}  \right)+C\|u\|_{p}.
\end{align*}

Now, summing for $h=1,2,3$ one has
\begin{align*}
&
(1-\sigma)^{3}\|a^{\frac{3}{2}}D^{3}u \|_{p,\sigma R}+(1-\sigma)^{2}\|aD^{2}u\|_{p,\sigma R}+(1-\sigma)\|a^{\frac{1}{2}}Du\|_{p,\sigma R}
 \leq C \|Au\|_{p,R}+C\|u\|_{p}\\
&\qquad +\frac{1}{2}\left( (1-\sigma')^{3}\|a^{\frac{3}{2}}D^{3}u \|_{p,\sigma'R}+(1-\sigma')^{2}\|aD^{2}u\|_{p,\sigma'R}
	+(1-\sigma')\|a^{\frac{1}{2}}Du\|_{p,\sigma'R}  \right).
\end{align*}
Set $\Phi(u)=\sup_{\sigma\in (0,1)}\left\{(1-\sigma)^{3}\|a^{\frac{3}{2}}D^{3}u \|_{p,\sigma R}+(1-\sigma)^{2}\|aD^{2}u\|_{p,\sigma R}+(1-\sigma)\|a^{\frac{1}{2}}Du\|_{p,\sigma R} \right\}$
then the above inequality reads as 
\[
\Phi(u)\leq \frac{1}{2}\Phi(u)+C\left( \|Au\|_{p,R}+\|u\|_{p}\right)
\]
and then
\begin{align*}
&\frac{1}{8}\left(
	\|a^{\frac{3}{2}}D^{3}u \|_{p,\frac{1}{2}R}+\| aD^{2}u\|_{p,\frac{1}{2} R}+\|a^{\frac{1}{2}}Du\|_{p,\frac{1}{2} R}
	\right) \\
&\quad \leq \frac{1}{8}\|a^{\frac{3}{2}}D^{3}u \|_{p,\frac{1}{2}R}+\frac 14 \|aD^{2}u\|_{p,\frac{1}{2} R}+\frac 12\|a^{\frac{1}{2}}Du\|_{p,\frac{1}{2} R} \\
&\quad \leq \Phi (u)\leq  C \|Au\|_{p,R}+C\|u\|_{p}.
\end{align*}
Therefore, for all $u \in W_{\rm loc}^{4,p}(\R^{N})$ we have
\begin{equation}\label{eq:cover-max}
\|a^{\frac{3}{2}}D^{3}u \|_{p,\frac{1}{2}R}+\| aD^{2}u\|_{p,\frac{1}{2} R}+\|a^{\frac{1}{2}}Du\|_{p,\frac{1}{2} R}
\leq C\left( \|Au\|_{p,R}+\|u\|_{p,R} \right)
\end{equation}
and by \eqref{eq:quartordine}
\begin{equation}\label{eq:cover-max2}
\|a^{2}D^{4}u\|_{p,\frac 12 R} \leq 
C\left (\|Au\|_{p,R}+  \|a^{\frac{3}{2}}D^{3}u \|_{p,R}+\| aD^{2}u\|_{p, R}+\|a^{\frac{1}{2}}Du\|_{p,R}
+\|u\|_{p,R} \right).
\end{equation}

Now let $u\in W_{\rm loc}^{4,p}(\R^{N})\cap L^{p}(\R^{N})$ such that $Au\in L^{p}(\R^{N})$ and
let $\{B_{\frac{1}{2}\rho(x_n )}(x_n)\}$ be a countable covering of $\R^N$ as in Proposition \ref{pr:covering}
such that at most $\zeta $ among the double balls $\{B_{\rho(x_n )}(x_n)\}$ overlap.

We write \eqref{eq:cover-max} with $x_0$ replaced by $x_n$ and summing over $n$ we obtain
\[
\|a^{\frac{3}{2}}D^{3}u \|_{p}+\| aD^{2}u\|_{p}+\|a^{\frac{1}{2}}Du\|_{p}
 \leq C\left( \|Au\|_{p}+\|u\|_{p} \right).
\]
Writing \eqref{eq:cover-max2} with $x_0$ replaced by $x_n$ and summing over $n$ we obtain
\begin{align*}
 \|a^{2}D^{4}u\|_{p}&\leq 
C\left(\|Au\|_{p}+
\|a^{\frac{3}{2}}D^{3}u \|_{p}+\| aD^{2}u\|_{p}+\|a^{\frac{1}{2}}Du\|_{p}+\|u\|_{p}
\right)	
\\
&  \leq C\left( \|Au\|_{p}+\|u\|_{p} \right)
\end{align*}
and therefore
\[\|u\|_{D(A_p)}\leq C(\|Au\|_p+\|u\|_p).\]

\end{proof}

\section{Generation results in $L^p(\R^N)$}
In this section we investigate the generation of an analytic semigroup in $L^p(\R^N)$ for $1<p<\infty$.
Following Agmon, we first prove an a priori estimate for the resolvent.

\begin{theorem} \label{agmon-comp}
Let  $1<p<\infty$ and assume that $a$ satisfies (\ref{gradient}). Then, if $\lambda=re^{i\theta}$ for some $r>r_{p}$ and $\theta\in [-\frac{\pi}{2},\frac{\pi}{2}]$, the following inequality holds for every $u\in D(A_p)$
$$|\lambda|\|u\|_p\leq C\|\lambda u-Au\|_p$$
where $C=C(N,p,c)$.
 \end{theorem}
\begin{proof}
Let $u\in C_{c}^{\infty}(\R^{N})$ and set $v(t,x)=\psi (t)e^{ir t}u(x)$ where $r>0$ and $\psi\in C_c^\infty(\R)$ is such that $0\leq\psi\leq1,\psi(t)=1$ for $t\in \left[-\frac{1}{2},\frac{1}{2}\right],\ \psi(t)=0$ for $|t|\geq 1$ and $|D^h\psi|\leq {C} ,h=1,2,3,4$.
By Proposition \ref{a-prioriAux} 
\begin{align*}
\sum_{k=1}^4 \|D^{(k)}_tv\|_{L^p(\R^{N+1})}&
	+\|a^\frac{1}{2}D v\|_{L^p(\R^{N+1})}+\|a D^2v\|_{L^p(\R^{N+1})}
	+\|a^\frac{3}{2}D^3 v\|_{L^p(\R^{N+1})}\\&+\|a^2D^4 v\|_{L^p(\R^{N+1})}
		\leq C(\|A_1v\|_{L^p(\R^{N+1})}+\|v\|_{L^p(\R^{N+1})}).
\end{align*}
Observe  that 
$$\|A_1 v\|_{L^p(\R^{N+1})}\leq\|-a^{2}\Delta ^2 u-e^{i\theta}r^4u\|_{L^p(\R^{N})}+C\sum_{k=0}^3r^{k}\|u\|_{L^p(\R^{N})}.$$
Therefore
\begin{align*}
\sum_{k=1}^4 &\|D^{(k)}_te^{irt}u\|_{L^p\left(\R^{N}\times\left[-\frac{1}{2},\frac{1}{2}\right]\right)}
	+\|a^\frac{1}{2}e^{irt}D u\|_{L^p\left(\R^{N}\times\left[-\frac{1}{2},\frac{1}{2}\right]\right)}\\
	&
	+\|a e^{irt}D^2u\|_{L^p\left(\R^{N}\times\left[-\frac{1}{2},\frac{1}{2}\right]\right)}
	+\|a^\frac{3}{2}e^{irt}D^3 u\|_{L^p\left(\R^{N}\times\left[-\frac{1}{2},\frac{1}{2}\right]\right)}\\
	&+\|a^2e^{irt}D^4 u\|_{L^p\left(\R^{N}\times\left[-\frac{1}{2},\frac{1}{2}\right]\right)}
	\leq C(\|-a^{2}\Delta ^2 u-r^4e^{i\theta}u\|_{L^p(\R^{N})}+\sum_{k=0}^3r^{k}\|u\|_{L^p(\R^{N})}).
\end{align*}
On the other hand
\begin{align*}
\sum_{k=1}^4 &\|D^{(k)}_te^{irt}u\|_{L^p\left(\R^{N}\times\left[-\frac{1}{2},\frac{1}{2}\right]\right)}= \sum_{k=1}^4 \|r^ku\|_{L^p\left(\R^{N}\right)}
\end{align*}
and then
\[
\left( r^{4}+(1-C)\sum_{k=0}^{3}r^{k} \right)\|u\|_{L^p\left(\R^{N}\right)}\leq C(\|-a^{2}\Delta ^2 u-r^4e^{i\theta}u\|_{L^p(\R^{N})}.
\]
Since 
$r^4-C\sum_{k=0}^3r^{k}\geq \frac{r^4}{2}$ for $r$ large enough, it follows that
$$r^4\|u\|_{L^p\left(\R^{N}\right)}\leq C\|-a^{2}\Delta ^2 u-r^4e^{i\theta}u\|_{L^p(\R^{N})}.$$ 
Setting $\lambda=r^{4}e^{i\theta}$ we have
$$|\lambda| \|u\|_{L^p\left(\R^{N}\right)}\leq C\|-a^{2}\Delta ^2 u-\lambda u\|_{L^p(\R^{N})}.$$
\end{proof}

The previous estimate leads to the solvability of the elliptic problem. 

\begin{theorem}\label{solv}
Let  $1<p<\infty$ and assume that $a$ satisfies (\ref{gradient}). Then there exists $r_p>0$ such that  the equation $\lambda u-Au=f$ is uniquely solvable in $D(A_p)$ for $f\in L^p(\R^N)$ and  $\Rp\lambda>r_p$.  
\end{theorem}
\begin{proof} Let $\sigma>0$ and define the function
$$a_\sigma(x)=\frac{a(x)}{1+\sigma a(x)}+\sigma,\quad x\in\R^N.$$ Observe that $\sigma\leq a_\sigma(x)\leq\sigma+\sigma^{-1}$, $|D a_\sigma|\leq c a_\sigma^\frac{1}{2}$ with  $c$ independent of $\sigma$. Consider the operators with bounded coefficients $A_\sigma=-(a_\sigma(x))^2\Delta^2$. 
By \cite[Theorem 3.2.2]{lun}, there exist $\omega_{\sigma,p}\in\R$, $M_{\sigma,p}>0$ such that $\{\lambda\in\C:\ \Rp \lambda\geq \omega_{\sigma,p}\}\subset\rho(A_{\sigma})$ and 
\begin{equation*} 
|\lambda|\|u\|_p\leq M_{\sigma,p}\|\lambda u-A_\sigma u\|_p
\end{equation*} for every $u\in W^{4,p}(\R^N))$  and $\lambda\in\C$ with $\Rp \lambda\geq \omega_{\sigma,p}$.
Morever, by applying Theorem  \ref{agmon-comp}  to $A_\sigma$, since the condition in (\ref{gradient}) is satisfied with a constant independent of $\sigma$, we have that there exists $r_p>0$  (independent of $\sigma$) such that  for every  $\lambda\in\C$ with $\Rp\lambda>r_p$ and for every $u\in C_c^\infty(\R^N)$, and then in $W^{4,p}(\R^N)$ by density, 
\begin{equation} \label{stima-ind}
|\lambda|\|u\|_p\leq C_p\|\lambda u-A_\sigma u\|_p
\end{equation}
for some constant $C_p$ independent of $\sigma$.  Let $\overline{\lambda}\in \rho (A_{\sigma })$ and $\Rp\overline{\lambda}>r_p$, then the inequality (\ref{stima-ind}) gives that
\begin{equation} \label{resolvent}
\|R(\overline{\lambda},A_\sigma)\|_p\leq \frac{C_p}{|\overline{\lambda}|}\leq \frac{C_p}{r_p}.
\end{equation}
By classical result, if $|\lambda-\overline{\lambda}|<\frac{1}{\|R(\overline{\lambda},A_\sigma)\|_p}$, then  $\lambda\in\rho (A_{\sigma })$. By (\ref{resolvent}), if
$|\lambda-\overline{\lambda}|<\frac{r_p}{C_p}$, then  $\lambda\in\rho (A_{\sigma })$.  
By covering with balls of radius $\frac{r_p}{C_p}$, it follows that, if $\Rp\lambda>r_p$, then $\Rp\lambda\in\rho(A_{\sigma })$.

Let $\Rp\lambda>r_p$, $f\in L^p(\R^N)$. Denote by $u_\sigma\in W^{4,p}(\R^N)$ the unique solution of $\lambda u-A_\sigma u=f$. Then $\{u_\sigma\}$ is bounded in $L^p(\R^N)$ and from $A_\sigma u_\sigma=\lambda u_\sigma -f$ we have that also $\{A_\sigma u_\sigma\}$ is bounded in $L^p(\R^N)$. Moreover, as in Proposition \ref{domain}, we have that
\begin{equation}\label{a-priori-sigma}\|a_\sigma^\frac{1}{2}D u_\sigma\|_p+\|a_\sigma D^2u_\sigma\|_p+\|a_\sigma^\frac{3}{2}D^3 u_\sigma\|_p+\|a_\sigma^2D^4 u_\sigma\|_p\leq C(\|A_\sigma u_\sigma\|_p+\|u_\sigma\|_p)\end{equation}  for some constant $C$ independent of $\sigma$. It follows that, since 
$$\frac{1}{a_\sigma}\leq \frac{1}{a}+\sigma\leq \frac{1}{a}+1$$ for $\sigma$ smaller than $1$ and $a$ is strictly positive in $\R^N$, for every positive $R$ there exists a positive constant $C_R$ independent of $\sigma$ such that
$$\|u_\sigma\|_{W^{4,p}(B_R)}\leq C_R(\|A_\sigma u_\sigma\|_{L^p(B_R)}+\|u_\sigma\|_{L^p(B_R)})\leq C_R\|f\|_{L^p(B_R)}.$$
 By using a compactness argument we can determine an infinitesimal sequence $\{\sigma_n\}$ and a function $u$ such that $u_{\sigma_n}$ converges weakly in $W^{4,p}(B_R)$ and strongly in $W^{3,p}(B_R)$ for every $R>0$. Moreover  $u_{\sigma_n}$, $D^{(k)}u_{\sigma_n}$ converge pointwise almost everywhere to $u$, $D^{(k)}u$ for every $1\leq k\leq 3$. By the equation, $a^2_{\sigma_n}\Delta^2u_{\sigma_n}$ converges to $a^2\Delta^2u$ locally in $L^p(\R^N)$. Letting $n$ to infinity  in (\ref{a-priori-sigma}) we deduce that
$$\|a^\frac{1}{2}D u\|_p+\|a D^2u\|_p+\|a^\frac{3}{2}D^3 u\|_p+\|a^2D^4 u\|_p\leq C\|f\|_p.$$
This proves that $u\in D(A_p)$ solves $\lambda u-Au=f$.
\end{proof}

In view of \cite[Proposition 3.2.8]{lor-rha-book} the operator $A$ is sectorial. Therefore, by standard generation results, cf. \cite[Ch. II, Thorem 4.6]{eng-nag}, we have proved the following generation theorem.
\begin{theorem}
Let  $1<p<\infty$ and assume that $a$ satisfies (\ref{gradient}). Then, $(A,D_p(A))$ generates an analytic $C_0$-semigroup in $L^p(\R^N)$.
\end{theorem}

The following inequality is a simple consequence that will be useful in the next section. 
\begin{coro}\label{agmon-comp-2}
Let  $1<p<\infty$ and assume that $a$ satisfies (\ref{gradient}). Then there exist $r_p>0$ and $C=C(N,p,c)$ such that  for every  $\Rp\lambda>r_p$ and $u\in D(A_p)$
\begin{equation*}
|\lambda|\|u\|_p+|\lambda|^\frac{3}{4}\|a^{\frac{1}{2}}Du\|_{p}+|\lambda|^\frac{1}{2}\|aD^{2}u\|_{p}+|\lambda|^\frac{1}{4}\|a^{\frac{3}{2}}D^{3}u\|_{p}+ \|a^2D^4u\|_p\leq C \|\lambda u-Au\|_p.
\end{equation*}
\end{coro}
\begin{proof} The proof follows by Lemma \ref{interpolation} and Theorem \ref{agmon-comp}.\end{proof}

\begin{remark}
We observe   that, by perturbation techniques, one can recover the same generation result for a more general operator containing lower order terms with unbounded coefficients satisfying growth bounds. Indeed, if the $h$ order coefficients, say $b_h(x)$,  satisfy  for every $h=1,2,3$
\[|b_h(x)|\leq k a(x)^{\frac{h}{2}}\qquad\forall\,x\in\R^N\]
for a positive constant $k$, then the resulting operator is, by Lemma \ref{interpolation}, $A$-bounded and one can apply \cite[Theorem 2.10]{eng-nag}.
\end{remark}

\section{Generation results in $C_{b}(\R^N)$ and $L^\infty(\R^N)$}

In this section we prove generation results in the spaces $C_b(\R^N)$ and $L^\infty(\R^N)$. To this purpose we assume   that there exist   some positive constants $\nu$ and $c$ such that
\begin{equation} \label{gradient-infty}\tag{$H_{\infty}$}
a  \in C^1(\R^N),\quad a(x)\geq \nu>0\ {\rm for\ all}\ x\in\R^N,\quad |D a(x)|\leq c a(x)^\frac{1}{2}.
\end{equation}

We recall the following Morrey embedding, cf. \cite[Lemma 2.1]{bro-opi}. 
\begin{lemma}[Morrey] If $p>N$, then $W^{1,p}(\R^{N})$ is continuously embedded in $C_{b}(\R^{N})$. 
Moreover
\begin{equation}\label{morrey}
\|g\|_{L^{\infty} (B_{R}(x_{0})) }\leq CR^{-\frac Np}\left( \|g\|_{L^{p} (B_{R}(x_{0}))}+R\|Dg\|_{L^{p} (B_{R}(x_{0}))} \right)
\end{equation}
for every $g\in W^{1,p}_{loc}(\R^{N})$, $x_{0}\in \R^N$ and $R>0$.	
\end{lemma}

As a consequence we prove the following result. 	
\begin{proposition}\label{prop:infty}
Let $p>N$ and assume that $a$ satisfies (\ref{gradient-infty}). Then there exists $ r'_p>0$ such that for every  $\Rp\lambda>r'_p$ and
$u\in C^{3}_{b}(\R^{N})\cap W^{4,p}_{\rm loc}(\R^{N})$ with $a^{\frac h2}D^{h}u\in C_{b}(\R^{N})$ for $h=0,1,2,3$ and $Au \in C_{b}(\R^{N})$, it holds
\begin{align*}
 \sum_{h=0}^{3}|\lambda|^{1-\frac h4} \|a^{\frac h2}D^{h}u\|_{\infty} 
	+\sup_{x\in \R^{N}} \frac  { |\lambda|^{\frac {N}{4p}}} {a(x)^{\frac{N}{2p}}} \|a^{2}D^{4}u\|_{ L^{p}(B_{R({x})  }(x))  }\leq C\|\lambda u-Au\|_{\infty}
\end{align*}
where $R({x})=\frac{a(x)^{\frac 12}}{|\lambda |^{\frac 14}}$.
\end{proposition}
\begin{proof}
Let $u$ be as in the hypotheses of the proposition,   $\Rp\lambda>r_p$ as in Corollary \ref{agmon-comp-2} and let $\delta>0$ be a constant which   will be chosen later.
Fix $x_{0}\in \R^{N}$ and set $R=R(x_{0})=\frac{a(x_{0})^{\frac 12}}{|\lambda |^{\frac 14}}$.
Let $\vartheta \in C_{c}^{\infty}(\R^{N})$ be such that $0\leq \vartheta \leq 1$, $\vartheta =1$ in $B_{R}(x_{0})$, $\vartheta=0$ in
$B^{c}_{R(1+\delta)} (x_{0})$ and $|D^{h}\vartheta|\leq {C}{R^{-h}\delta ^{-h}} $ for $h=1,2,3,4$  for some positive constant $C$ which does not depend on $x_{0}$ and $\delta.$ 
By condition $||D a^{\frac{1}{2}}||_{\infty}\leq \frac{c}{2}$ 
we have that
\[
a(x)^{\frac 12}\leq \left( 1+\frac c2\frac{1+\delta}{| \lambda|^{\frac 14}} \right)a(x_{0})^{\frac 12}=\kappa a(x_{0})^{\frac 12}
\text{ for every }x\in B_{R(1+\delta)}(x_{0})
\]
and then for $h=1,2,3,4$
\[
a(x)^{\frac h2}\leq \kappa^{h}a(x_{0})^{\frac h2} \text{ for every }x\in B_{R(1+\delta)}(x_{0});
\] 
moreover, choosing $|\lambda|\geq c^{4}$ we have
\[
a(x_{0})^{\frac 12}\leq \left(1-\frac{c}{2}\frac{1}{|\lambda|^{\frac 14}}\right)^{-1}a(x)^{\frac{1}{2}}\leq 2 a(x)^{\frac{1}{2}}
\text{ for every }x\in B_{R}(x_{0})
\]
and then for $h=1,2,3,4$
\[
a(x_{0})^{\frac h2}\leq 2^{h} a(x)^{\frac{h}{2}}\text{ for every }x\in B_{R}(x_{0}).
\]
Now for $h=0,1,2,3$ from \eqref{morrey} we have
\begin{align*}
 |\lambda| ^{1-\frac{h}{4}}|a(x_{0})^{\frac{h}{2}}&D^{h}u(x_{0})|
	\leq CR^{-\frac Np}\left( |\lambda| ^{1-\frac{h}{4}}|a(x_{0})^{\frac{h}{2}}| \|D^{h}u\|_{p,R} 
	+ |\lambda| ^{1-\frac{h}{4}}|a(x_{0})^{\frac{h}{2}}| R\|D^{h+1}u\|_{p,R} \right) \\
&\quad \leq CR^{-\frac Np}\left( |\lambda| ^{1-\frac{h}{4}}|a(x_{0})^{\frac{h}{2}}| \|D^{h}u\|_{p,R} 
	+ |\lambda| ^{1-\frac{h+1}{4}}|a(x_{0})^{\frac{h+1}{2}}| \|D^{h+1}u\|_{p,R} \right) \\
&\quad \leq 8CR^{-\frac Np}\left( |\lambda| ^{1-\frac{h}{4}} \| a^{\frac{h}{2}}D^{h}u\|_{p,R} 
	+ |\lambda| ^{1-\frac{h+1}{4}}\|a^{\frac{h+1}{2}}D^{h+1}u\|_{p,R} \right). 	
\end{align*}
Summing over $h=0,1,2,3$ we have
\begin{equation*}
\sum_{h=0}^{3}|\lambda| ^{1-\frac{h}{4}}|a(x_{0})^{\frac{h}{2}}D^{h}u(x_{0})|
	\leq CR^{-\frac Np}\sum_{h=0}^{4}  |\lambda| ^{1-\frac{h}{4}} \| a^{\frac{h}{2}}D^{h}u\|_{p,R} 
\end{equation*}
and also
\begin{equation}\label{eq:sum-lifnty1}
\sum_{h=0}^{3}|\lambda| ^{1-\frac{h}{4}}|a(x_{0})^{\frac{h}{2}}D^{h}u(x_{0})|
		+\frac{|\lambda |^{\frac {N}{4p}}}{a(x_{0})^{\frac {N}{2p}}}\|a^{2}D^{4}u\|_{p,R}
	\leq CR^{-\frac Np}\sum_{h=0}^{4}  |\lambda| ^{1-\frac{h}{4}} \| a^{\frac{h}{2}}D^{h}u\|_{p,R}.
\end{equation}
Now, setting $v=\vartheta u$, from Corollary \ref{agmon-comp-2} we have
\begin{align*}
 \sum_{h=0}^{4}  |\lambda| ^{1-\frac{h}{4}} \| a^{\frac{h}{2}}&D^{h}u\|_{p,R}\leq 
	\sum_{h=0}^{4}  |\lambda| ^{1-\frac{h}{4}} \| a^{\frac{h}{2}}D^{h}v\|_{p,R(1+\delta)}\\
&\quad \leq C\|\lambda v-Av\|_{p}\\
&\quad \leq C\left( \| \lambda u-Au\|_{p,R(1+\delta)} 
	+\sum _{h=0}^{3}\|a^{2}D^{h}uD^{4-h}\vartheta\|_{p,R(1+\delta)} \right)\\
&\quad \leq C\left( \| \lambda u-Au\|_{p,R(1+\delta)} 
	+\sum _{h=0}^{3} \frac{1}{R^{4-h}\delta ^{4-h}}    \kappa^{4-h} a(x_{0})^{2-\frac h2}\|a^{\frac{h}{2}}D^{h}u\|_{p,R(1+\delta)} \right)\\
&\quad \leq C\left( \| \lambda u-Au\|_{p,R(1+\delta)} 
	+\sum _{h=0}^{3} \frac{\kappa^{4-h}}{\delta ^{4-h}}|\lambda |^{1-\frac h4}     \|a^{\frac{h}{2}}D^{h}u\|_{p,R(1+\delta)} \right)\\
&\quad \leq C(1+\delta)^{\frac Np} R^{\frac Np}\left( \| \lambda u-Au\|_{\infty} 
	+\sum _{h=0}^{3} \frac{\kappa^{4-h}}{\delta ^{4-h}}|\lambda |^{1-\frac h4}     \|a^{\frac{h}{2}}D^{h}u\|_{\infty} \right).
\end{align*}
Then, by \eqref{eq:sum-lifnty1}, we have
\begin{align*}
\sum_{h=0}^{3}|\lambda| ^{1-\frac{h}{4}}&|a(x_{0})^{\frac{h}{2}}D^{h}u(x_{0})|
		+\frac{|\lambda |^{\frac {N}{4p}}}{a(x_{0})^{\frac {N}{2p}}}\|a^{2}D^{4}u\|_{p,R} \\
&\quad \leq C(1+\delta)^{\frac Np} \left( \| \lambda u-Au\|_{\infty} 
	+\sum _{h=0}^{3} \frac{\kappa^{4-h}}{\delta ^{4-h}}|\lambda |^{1-\frac h4}     \|a^{\frac{h}{2}}D^{h}u\|_{\infty} \right).
\end{align*}
Taking the supremum over $x_{0}\in \R^{N}$ we have
\begin{align*}
\sum_{h=0}^{3}|\lambda| ^{1-\frac{h}{4}}&\|a^{\frac{h}{2}}D^{h}u\|_{\infty}
		+\sup_{x_{0}\in \R^{N}}\frac{|\lambda |^{\frac {N}{4p}}}{a(x_{0})^{\frac {N}{2p}}}\|a^{2}D^{4}u\|_{p,R} \\
&\quad \leq C(1+\delta)^{\frac Np} \left( \| \lambda u-Au\|_{\infty} 
	+\sum _{h=0}^{3} \frac{\kappa^{4-h}}{\delta ^{4-h}}|\lambda |^{1-\frac {h}{4}}     \|a^{\frac{h}{2}}D^{h}u\|_{\infty} \right).
\end{align*}
Now choose $\delta_{0}>1$ such that $C\frac{(1+\delta_{0})^{\frac Np}}{\delta_{0}}\leq \frac 14$ and  $r'_{p}\geq \max \{c^{4},r_{p}\}$ so that if $|\lambda|\geq r'_{p}$ we have
$\kappa^{4}=\left( 1+\frac c2\frac{1+\delta_{0}}{|\lambda|^{\frac 14}}\right)^{4}\leq 2$ and
$C(1+\delta_{0})^{\frac Np}\frac {\kappa^{4-h}}{\delta_{0}^{4-h}}\leq C(1+\delta_{0})^{\frac Np}\frac {\kappa^{4}}{\delta_{0}}\leq \frac 12$. This leads to the desired result.

\end{proof}

Let us now define 
\[
D(A_{0})=\left\{ u\in \underset{1\leq p <\infty}{\bigcap} W^{4,p}_{\rm loc}(\R^{N})\cap C_b(\R^N)\,|\,
		a^{\frac h2}D^{h}u, a^2\Delta^2u\in C_{b}(\R^{N}), h=0,\dots, 3 \right\}.
\]
We obtain the following generation result in $C_b(\R^N)$.
\begin{theorem}\label{generatcb}
Assume that $a$ satisfies (\ref{gradient-infty}). Then $(A_{0},D(A_{0}))$ generates an analytic $C_0$-semigroup in 
$C_{b}(\R^N)$.
\end{theorem}
\begin{proof}
We prove that $(A_{0},D(A_{0}))$ is sectorial.
First we prove that the resolvent set is non empty.

Let $f\in C_{b}(\R^{N})$ and $f_{n}=\varphi_{n}f$ where $\varphi_{n}\in C_{c}^{\infty}(\R^{N})$ is such that 
$0\leq \varphi_{n}\leq 1$, $\varphi_{n}=1$, in $B_{n}$ and $\varphi_{n}=0\in B^{c}_{2n}$.
Let $p>N$, since $f_{n}\in L^{p}(\R^{N})$ then, by Theorem \ref{solv}, for $Re\lambda > r_{p} $
there exists a unique $u_{n}\in D(A_{p})$ such that
\begin{equation}\label{eq:solvn}
\lambda u_{n}-Au_{n}=f_{n}.
\end{equation}
We have, in particular, $a^{\frac h2}D^{h}u_{n}\in W^{1,p}(\R^{N})$ for $h=0,1,2,3$.
Indeed, by hypothesis (\ref{gradient-infty}) we have $1\leq \frac{1}{\nu^{\frac 12}}a(x)^{\frac 12}$
and $|Da(x)|\leq ca(x)^{\frac 12}$, then
\begin{align*}
&\left |D\left( a(x)^{\frac h2}D^{h}u_{n}(x) \right)\right |\\
&\qquad \leq C
\left(
\left |a(x)^{\frac h2-1}Da(x)D^{h}u_{n}(x) \right | +\left |a(x)^{\frac h2}D^{h+1}u_{n}(x) \right |
\right)\\
&\qquad \leq C\left(
\left |a(x)^{\frac h2-\frac 12}D^{h}u_{n}(x)\right | +\left |a(x)^{\frac h2}D^{h+1}u_{n}(x)\right |
\right) \\
&\qquad
\leq C\left(
\left |a(x)^{\frac h2}D^{h}u_{n}(x) \right | +\left |a(x)^{\frac {h+1}{2}}D^{h+1}u_{n}(x)\right |
\right)
\end{align*}
which is in $L^{p}(\R^{N})$ since $u_{n}\in D(A_{p})$.
Then $a^{\frac h2}D^{h}u_{n}\in C_{b}(\R^{N})$ for $h=0,1,2,3$ and we can apply
Proposition \ref{prop:infty} and obtain
\begin{align}\label{eq:stima-infty}
 \sum_{h=0}^{3}&|\lambda|^{1-\frac h4} \|D^{h}u_{n}\|_{\infty} \\
& \quad \leq C\left (\sum_{h=0}^{3}|\lambda|^{1-\frac h4} \|a^{\frac h2}D^{h}u_{n}\|_{\infty} 
	+\sup_{x\in \R^{N}} \frac  { |\lambda|^{\frac {N}{4p}}} {a(x)^{\frac{N}{2p}}} \|a^{2}D^{4}u_{n}\|_{ L^{p}(B_{R({x})  }(x))  }\right )\nonumber\\
&\quad	\leq C\|f_{n}\|_{\infty}\leq C\|f\|_{\infty}.\nonumber
\end{align}
Then $\{u_{n}\}$ is a  bounded sequence in $C^{3}(\R^{N})$ and in $W^{4,p}_{\rm loc}(\R^{N})$.
Moreover, there exists a subsequence such that $u_{k_{n}}$ converges to a function $u\in C^{2}_{b}(\R^{N})\cap W^{4,p}(\R^{N})$, locally uniformly in $\R^{N}$ and weakly in $W^{4,p}_{\rm loc}(\R^{N})$ and then strongly in $W^{3,p}_{\rm loc}(\R^{N})$. Therefore, up to a subsequence, $D^{3}u_{n}$ converge pointwise to $D^{3}u$.
Taking the weak limit in \eqref{eq:solvn} we have that $u$ solves weakly the equation $\lambda u-Au=f$.
Furthermore, taking the limit in \eqref{eq:stima-infty} it follows that $a^{\frac h2}D^{h}u\in C_{b}(\R^N)$ for $h=0,1,2,3$
and $ |\lambda|\, \|u\|_{\infty}\leq C\|f\|_{\infty}$.
Finally, since $f\in L^{q}_{\rm loc}(\R^{N})$ for every $q>1$, by elliptic regularity, $u\in W^{4,p}_{\rm loc}(\R^{N})$ and then
$u\in D(A_{0}).$
\end{proof}

The proof of Theorem \ref{generatcb} carries over to the $L^\infty$ setting, therefore the following theorem holds. 
 
\begin{theorem}
Assume that $a$ satisfies (\ref{gradient-infty}). Then $(A_{\infty},D(A_{\infty}))$, where
\[D(A_\infty)=\left\{u\in \underset{1\leq p <\infty}{\bigcap} W^{4,p}_{\rm loc}(\R^N)\cap L^\infty(\R^N):a^{\frac h2}D^{h}u,a^2\Delta^2u\in L^\infty(\R^{N}), h=0,\dots, 3 \right\}\] 
generates an analytic $C_0$-semigroup.
\end{theorem}

\bibliographystyle{amsplain}

\bibliography{bibfile}

\end{document}